\documentclass[11pt]{amsart}
\usepackage{amsmath}
\usepackage{amsthm}
\usepackage{amssymb,amscd}

\newtheorem{thm}          {Theorem}      [section]

\newtheorem{lemma}  [thm] {Lemma}

\newtheorem{theorem}       [thm]   {Theorem}     

\newtheorem{corollary}    [thm] {Corollary}

\theoremstyle{remark}

\DeclareMathOperator\diag{diag}
\DeclareMathOperator\ZZ{\mathbb Z}
\DeclareMathOperator\FF{\mathbb F}

\DeclareMathOperator\PSL{PSL}
\DeclareMathOperator\PGL{PGL}
\DeclareMathOperator\SL{SL}

\DeclareMathOperator\GL{GL}

\renewcommand\b\bar

\begin{document}

\title[Cosets of Sylow $p$-subgroups] 
{Cosets of Sylow $p$-subgroups and a Question of Richard Taylor}
\author{Daniel Goldstein}
\address{Center for Communications Research, 
San Diego, CA 92121}
\makeatletter\email{danielgolds@gmail.com}\makeatother
\author{Robert Guralnick}
\address{Department of Mathematics, University of
  Southern California, Los Angeles, CA 90089-2532, USA}
\makeatletter\email{guralnic@usc.edu}\makeatother

\dedicatory{For John Thompson on the occasion of his 80th birthday}

\subjclass[2010]{Primary 20D20, 20D06;  Secondary 11F80, 11R34}

\thanks{Guralnick was partially supported by the NSF
  grant DMS-1001962 and  Simons Foundation fellowship 224965.}

\begin{abstract}
We prove that for any odd prime $p$, there exist infinitely
many finite 
simple groups $S$ containing a Sylow $p$-subgroup $P$ of $S$ such that
some coset $gP$ of $P$ in $S$ consists of elements whose
order is divisible by $p$.  This allows us to answer a question
of Richard Taylor related to whether certain Galois representations
are automorphic.
\end{abstract}

\maketitle


\section{Introduction}

In 1967, answering a question of Lowell Paige, 
John Thompson \cite{jgt} proved the following result:

\begin{theorem} \label{paige}  Let $G=\PSL_2(53)$.
Let $Q$ be a Sylow $2$-subgroup of $G$ (of order $4$).
There exists $g \in G$ such that every element in
$gQ$ has even order.
\end{theorem}

Thompson actually worked in $\SL_2(q)$. 
He wrote down a system of equations
over $\FF_q$ with $q \equiv \pm 3 \mod 8$ and observed
that the existence of a coset of a Sylow $2$-subgroup
consisting of elements of even order amounted to finding
a solution to this system of equations.  He then observed
that for $q = 53$, there was a solution.  
Now this can easily be checked in MAGMA \cite{M} and indeed
$53$ can be replaced by $27$ and also by any $q
\ge 53$ with $q \equiv \pm 3 \mod 8$ (by showing that
the variety defined by Thompson always has points for
$q \ge 53$).     We have no knowledge of  Paige's motivation.

Here we prove an analog for $p$ odd.  Our motivation
for considering this problem was a question posed
by Richard Taylor regarding what are called adequate
groups.  See \cite{G, GH, T} for more about this question
and its relation to Galois representations which are automorphic. 

We actually prove the following theorem:

\begin{theorem} \label{main}  Let $p$ be an odd prime.
Let $D$ be the dihedral group of order $2p$. 
Let  $q$ be a prime power with 
$p$ dividing $q-1$ and $p^2$ not dividing $q-1$.
View $D$ as a subgroup of $G:=\PSL_2(q)$. 
If $q$ is sufficiently large, there exists  $g \in G$
such that every element of $gD$ has order divisible
by $p$.
\end{theorem}

Our method is similar to the one used by Thompson.
The idea of the proof is to write down a variety parametrizing
a family of solutions to the condition that every element
of $gD$ has order a multiple of $p$.  We then prove
that this variety is irreducible and defined over $\FF_q$
(indeed, defined over $\ZZ[\theta]$ where $\theta$ is a 
$p$th root of $1$).   We can then apply Lang-Weil \cite{LW} to 
conclude that there are solutions as long as $q$ is sufficiently large. 

In fact, there should be many more examples.  While
$gD$ is far from being a random subset of $G$, if one
assumes that this is the case, a straightforward computation
of probabilities shows that there should be many elements
$g \in G$ with $gD$ consisting of elements of order a multiple
of $p$.  On the other hand, one does need to be a bit careful.
If $G$ is a $p$-solvable group and $P$ is a Sylow $p$-subgroup of
$G$, then every coset $gP$ contains some $p'$-element
(let $H$ be a $p$-complement in $G$, then $|H \cap gP| =1$
for all $g \in G$). Our phenomenon may be related to the
degree of the smallest nontrivial character of $G$ (and
the relative size of the subset being considered). 
See \cite{gowers} for some results of this nature. 

We also note that a straightforward computation using MAGMA shows
that:

\begin{theorem} \label{char 2}
Let $G=\PSL_2(139)$.  Let $Q$ be a Sylow $2$-subgroup
of $G$ and let $N:= N_G(Q)$.  Then $N \cong A_4$.
There exists a coset
of $N$ in $G$ in which all elements have even order. 
\end{theorem}

The consequence of the above theorems is the following result
in \cite{G} which answers the question of Taylor.

\begin{corollary}  Let $p$ be prime.  
Let $k$ be an algebraically closed field of characteristic $p$.
There exists a finite group $G$ and an absolutely irreducible
faithful $kG$-module $V$ such that 
\begin{enumerate}
\item $H^1(G,k)=0$;
\item $p$ does not divide $\dim V$;
\item $H^1(G, V \otimes V^*)=0$; and
\item $\mathrm{End}(V)$ is not spanned by the images
of the $p'$-elements of $G$.
\end{enumerate}
\end{corollary}

It is worth noting that all such examples constructed
in \cite{G} have the property that $V$ is an induced
module.  It is also shown in \cite{G} that if $G$
is $p$-solvable, then (2) implies
that the $p'$-elements do span $\mathrm{End}(V)$.
In \cite{GH}, it was shown that if $p \ge 2 \dim V +2$,
then (1)-(4) hold. 

\section{Field Extensions}

We first recall a standard consequence of Kummer theory.

\begin{lemma} \label{kummer}
 Let $p$ be a prime.  Let $E/F$ be an extension
of fields with characteristic not $p$.  Assume that $E/F$
is Galois with elementary abelian Galois group $G$ of order
$p^m$ and that $F$ contains the $p$th roots of unity.
Then $(E^*/F^*)[p]$ has order $p^m$.
\end{lemma}

\begin{theorem} \label{Galois}
Let $D$ be a UFD with quotient field $F$
of characteristic not dividing $2p$ for some odd prime $p$.
Assume that $D$ contains $\theta$, a primitive $p$th root of $1$.
Let  $f_1, \ldots, f_m \in D$ with the $f_i$ square free
non-units
and pairwise having no common irreducible factors.    Fix an algebraic
closure $L$ of $F$.   Let $q_i \in L$ with $q_i^2 = f_i$.
Let $c_i \in L$ with $c_i^p \in D[q_i]$ and $c_i$ not in
$D[q_i]$.  Let $\Omega$ be the subset of $F[c_1, \ldots, c_m]$
consisting of elements of the form  
$\prod_i q_i^{e_i} \prod_i c_i^{\ell_i}$ where
$e_i \in \{0, 1\}$ and $\ell_i \in \{0, 1, \ldots, p-1\}$. 
\begin{enumerate}
\item  $[F[c_1, \ldots c_m]:F] = (2p)^m$;
\item  $D[c_1, \ldots, c_m]$ is a free $D$-module of rank
$(2p)^m$ with basis $\Omega$.
\end{enumerate} 
\end{theorem}

\begin{proof} To prove (1),  we 
may invert all irreducible elements of 
$D$ other than those dividing $f:=f_1 \ldots f_m$ and 
so assume that $D$ has only finitely many height $1$
prime ideals.  This implies that $D$ is a semilocal PID
(to see this, it suffices to reduce to the case that $D$
is a local UFD with a single irreducible element where
the result is obvious).   Indeed, by inverting some further
prime elements, we may assume that the $f_i$ are distinct
primes.  Let $v_i$ be the $f_i$ valuation on $D$.

Let $Q=F[q_1, \ldots, q_m]$.  By Lemma \ref{kummer} $Q$ is Galois
of degree $2^m$.  So it suffices to show
that $L:=F[c_1, \ldots c_m]$ has degree $p^m$ over $Q$.
Let $c = c_i$.  Since $c^p \in Q$ and $Q$ contains
the $p$th roots of $1$,  for any field $L'$ containing
$Q$, either $c \in L'$ or the minimal polynomial of 
$c$ over $L'$ is $x^p - c^p$.

This implies that $L/Q$ is Galois of degree $p^s$
for some $s \le m$ and moreover, the Galois group
is an elementary abelian $p$-group.

Let $0 \le e_i < p$ be integers for $i=1, \ldots, m$.
Extend each valuation $v_i$ to a valuation $w_i$ on
$L$.  Note that $w_j(\prod c_i^{e_i}) > 0$ for some
$j$ unless $e_i=0$ for all $i$.  Thus, 
the elements $\prod c_i^{e_i}$ are all distinct
modulo $F^*$. 

Thus, the $p$-torsion subgroup of $L^*/Q^*$ has order
at least $p^m$.  Now apply the previous lemma.  
This shows that $[F[c_1, \ldots c_m]:F]=(2p)^m$.
Clearly, $\Omega$  is a spanning set and therefore
a basis for $F[c_1, \ldots c_m]/F$.  
Thus $D[c_1, \ldots, c_m] = D[\Omega]$ is free over $D$
with basis $\Omega$.
\end{proof}

\begin{theorem} Let $p$ be a prime.
Let $F$ be a field of characteristic not dividing $2p$
containing a nontrivial $p$th root of unity $\theta$.
Let $I$ be the ideal of $R:=F[a,b,c_1, \ldots, c_p]$
generated by
$$
(c_i^p\theta)^2 - (a \theta^i + b \theta^{-i})(c_i^p\theta) + 1,
$$
for $i = 1, \ldots, p$.
Then $I$ is a prime ideal of $R$. 
\end{theorem}

\begin{proof}   Set $d_i = \theta c_i^p$.
Thus, 
$ d_i^2 -( a \theta^i + b \theta^{-i})d_i +1 \in I. $

Note that the discriminant of this quadratic is
$$
f_i:=(a \theta^i + b \theta^{-i})^2 - 4
= (a \theta^i + b \theta^{-i} +2)(a \theta^i + b \theta^{-i} - 2).
$$

Note that $f_i$ are pairwise relative prime in $D:=F[a,b]$.
Let $\Omega$ be the set of elements
of the form $\prod_i d_i^{e_i} \prod_i c_i^{\ell_i}$
where $e_i \in \{0, 1\}$ and $\ell_i \in \{0, 1, \ldots, p-1\}$.
Applying the previous result shows that
$\Omega$ is a basis for $R/I$ over $D$ since they are linearly
independent in a homomorphic image.  

Now $(R/I \otimes_D Q(D))$ surjects onto the field described
in the previous theorem.  Comparing ranks show that 
$R/I \otimes_D Q(D)$ is a field and $R/I$ injects into this field,
whence $I$ is a prime ideal as required.  
\end{proof}

\section{Cosets of Subgroups}

We can now prove an analog of Thompson's theorem
for odd primes.   

\begin{theorem} \label{thompson} 
Let $p$ be an odd prime.  Let $q$ be a 
prime power such that:
\begin{enumerate}
\item $q \equiv 1 \pmod p$;
\item $p^2$ does not divide $q-1$.
\end{enumerate} 
Let $G=\PSL_2(q)$.
Let $P$ be a Sylow $p$-subgroup of $G$ (of order $p$). 
Let $P \le D$ be a dihedral group of order $2p$ in $G$.
If $q$ is sufficiently large, there exists $g \in \PGL_2(q)$
such that every element in $gD$ has order divisible by $p$.
\end{theorem}

\begin{proof} It suffices to work in $H=\SL_2(q)$.
  Let $\theta$ be a nontrivial $p$th root of
$1\in \FF_q$.  We may assume that $P$ is generated
by the diagonal matrix $\diag(\theta, \theta^{-1})$
and that $D =\langle P, x \rangle$ where 
$$
x = \begin{pmatrix}  0 & 1 \\
                    -1 & 0 \\   \end{pmatrix}.
$$

Let $g \in H$ where 
$$
g = \begin{pmatrix}  a &-c \\
                     d & b \\   \end{pmatrix}.
$$

Note that an element of $G$ has order a multiple of $p$
if and only if its trace can be written as 
$e^p \theta^j + (e^p \theta^j)^{-1}$ for some nonzero $e \in \FF_q$
and some $j, 0 < j < p$  (an element of order a multiple of $p$
can be conjugated so that it commutes with $P$; i.e. it is contained
in the subgroup of diagonal matrices  -- this subgroup is 
of the form $P \times Q$ where $Q$ has order prime to $p$ -- in 
particular, every element in $Q$ is a $p$th power).

Consider the following $2p$ equations:
\begin{align*}
a \theta^i + b \theta^{-i} & =  (c_i^p \theta) + (c_i^p \theta)^{-1}\\
c \theta^i + d \theta^{-i} & =  (d_i^p \theta) + (d_i^p \theta)^{-1}
\end{align*}
for $i =1, \ldots, p$.  Then 
$$
g = \begin{pmatrix} a & -c \\
                    d & b \\   \end{pmatrix}
$$
has the property every element of $gD$ has order a multiple of $p$.

By the previous result, these equations define a variety which is a
direct product of two $2$-dimensional irreducible varieties
defined over any field containing the $p$th roots of $1$.    Thus, by 
Lang-Weil \cite{LW}, for $q$ sufficiently large, there are solutions
(with $ab + cd \ne 0$).
\end{proof}

Since $\PGL_2(q) \le PSL_2(q^2)$, replacing $q$ by $q^2$ 
(which still satisfies the various 
hypotheses),   and so
we can in fact take $g \in \PSL_2(q)$ if we wish.  
Alternatively, we can take our two varieties which we will denote
$X(a,b)$ and $X(c,d)$.  Let $Y$
be the subvariety of  $X(a,b) \times X(c,d) \times \mathbb{A}^1$
defined by $ab + cd = w^2$ 
(where $w$ is the parameter on $\mathbb{A}^1$).
It is easy to see that (since $X(a,b)$ and $X(c,d)$
are irreducible), so is $Y$ and so Lang-Weil applies to $Y$
as well. Thus, we can choose our $g$ above to have square determinant,
whence reducing modulo the center of $\GL_2(q)$,  the image
of $g$ is in $\PSL_2(q)$.

\end{document}